\documentclass[a4paper, 12pt]{amsart}
\usepackage{amssymb}
\usepackage{amsthm}
\usepackage{amsmath}
\usepackage[numeric]{amsrefs}
\usepackage{hyperref}   
\usepackage{graphicx,color}

\newcommand{\fun}{\mathbb{F}_1}
\renewcommand{\Re}{\mathop{\mathrm{Re}}}
\newcommand{\sts}[1]{\mathcal{O}_{#1}}
\newcommand{\Abel}{\mathfrak{Ab}}
\newcommand{\Mo}{\mathfrak{M}_0}
\newcommand{\Alg}{\mathfrak{Alg}}
\newcommand{\bbA}{\mathbb{A}}
\newcommand{\bbF}{\mathbb{F}}
\newcommand{\bbN}{\mathbb{N}}
\newcommand{\bbZ}{\mathbb{Z}}
\newcommand{\bbQ}{\mathbb{Q}}
\newcommand{\bbR}{\mathbb{R}}
\newcommand{\bbC}{\mathbb{C}}
\newcommand{\bbP}{\mathbb{P}}
\newcommand{\Hom}{\mathop{\mathrm{Hom}}\nolimits}
\newcommand{\Spec}{\mathop{\mathrm{Spec}}}
\newcommand{\MSpec}{\mathop{\mathrm{spec}}}
\newcommand{\Kzeta}{\zeta^{\mathrm{K}}}
\newcommand{\CPoly}{\mathfrak{C}}
\newcommand{\FPoly}{\mathfrak{F}}
\newcommand{\PuiPoly}{\bbR[t^{1/\infty}]}
\newcommand{\primes}{\mathbb{P}}
\newcommand{\PrimePow}{\mathbb{P}^\mathbb{N}}

\newtheorem{theorem}{Theorem}[section]
\newtheorem{proposition}[theorem]{Proposition}
\newtheorem{lemma}[theorem]{Lemma}
\newtheorem{corollary}[theorem]{Corollary}

\theoremstyle{definition}
\newtheorem{definition}[theorem]{Definition}
\newtheorem{eg}[theorem]{Example}
\newtheorem{remark}[theorem]{Remark}

\title[ceiling/floor Puiseux polynomial]{Absolute zeta functions arising from ceiling and floor Puiseux polynomials}
\date{}

\author{Yoshinosuke Hirakawa}
\address[Yoshinosuke Hirakawa]{Graduate School of Sciences and Technology for Innovation \\ Yamaguchi University, 1677-1, Yoshida, Yamaguchi, Japan}
\email{yhirakawa@yamaguchi-u.ac.jp, hirakawa@keio.jp}
\thanks{This research is supported by JSPS KAKENHI Grant Numbers JP21K13779, JP22J10658 and JP22KJ2684.}

\author{Takuki Tomita}
\address[Takuki Tomita]{Department of Mathematics, Faculty of Science and Technology, Keio University, 3-14-1 Hiyoshi, Kouhoku-ku, Yokohama, 223-8522, Japan}
\email{takuki@keio.jp}

\subjclass[2020]{14G10 (Primary), 11M41, 11R59 (Secondary)}

\keywords{absolute zeta functions, monoid schemes, elliptic curves}

\begin{document}
\maketitle

\begin{abstract}
For the $\mathbb{Z}$-lift $X_\mathbb{Z}$ of a monoid scheme $X$ of finite type, Deitmar-Koyama-Kurokawa calculated its absolute zeta function by interpolating $\#X_\mathbb{Z}(\mathbb{F}_q)$ for all prime powers $q$ using the Fourier expansion. This absolute zeta function coincides with the absolute zeta function of a certain polynomial. In this article, we characterize the polynomial as a ceiling polynomial of the sequence $\left(\#X_\mathbb{Z}(\mathbb{F}_q)\right)_q$, which we introduce independently. Extending this idea, we introduce a certain pair of absolute zeta functions of a separated scheme $X$ of finite type over $\mathbb{Q}$ by means of a pair of Puiseux polynomials which estimate ``$\#X(\mathbb{F}_{p^m})$'' for sufficiently large $p$. We call them the ceiling and floor Puiseux polynomials of $X$. In particular, if $X$ is an elliptic curve, then our absolute zeta functions of $X$ do not depend on its isogeny class.
\end{abstract}

%
%
%
%
%
\section{Introduction}\label{section:Intro}
%
%
%
%
%

In number theory, it is traditionally important to study the solutions over $\bbZ$ of algebraic equations. One of the approaches to such a problem is to investigate the set $\mathcal{X}(\bbF_{p^m})$ of the $\bbF_{p^m}$-rational points of a scheme $\mathcal{X}$ of finite type over $\bbZ$ and unify information on $\mathcal{X}(\bbF_{p^m})$'s. In particular, the zeta function of $\mathcal{X}_{\bbF_p}:=\mathcal{X}\times\Spec\bbF_p$ defined by
\[Z(\mathcal{X}_{\bbF_p}, p^{-s}):=\exp\left(\sum_{m=1}^\infty\frac{\#\mathcal{X}(\bbF_{p^m})}{m}p^{-ms}\right)\]
has been studied as exemplified by the Weil Conjecture since the 20th century.

Soul\'e~\cite{soule2004} studied ``the limit of $Z(\mathcal{X}_{\bbF_p}, p^{-s})$ as $p\to 1$" when there exists a polynomial $f_{\mathcal{X}}(t)=\sum_{j=0}^R a_j t^j$ satisfying that $\#\mathcal{X}(\bbF_{p^m})=f_{\mathcal{X}}(p^m)$ for any prime number $p$ and $m\in\bbN$. More precisely, he found the fact that
\begin{equation}\label{eq: Soule-limit}
    \lim_{p\to 1}(p-1)^{f_{\mathcal{X}}(1)}\exp\left(\sum_{m=1}^\infty\frac{f_{\mathcal{X}}(p^m)}{m}p^{-ms}\right)=\prod_{j=0}^R (s-j)^{-a_j} \tag{S}
\end{equation}
and called it the \emph{absolute zeta function} of $\mathcal{X}$.
Later, Deitmar~\cites{Deitmar2005, Deitmar2006} introduced a \emph{monoid scheme} and realized the above rational function as an invariant of a monoid scheme.
After Deitmar's work, Connes and Consani generalized the above definition of absolute zeta functions as follows.

\begin{definition}[\cites{Connes-Consani_Compositio}]
Let $f\colon[1,\infty)\to\bbC$ be a function satisfying that $|f(t)|\leq Ct^d$ for some $C>0$ and $d>0$. Then, the \emph{absolute zeta function} of $f$ is defined by the limit
\[\zeta_f(s):=\lim_{p\to 1+}(p-1)^{f(1)}\exp\left(\sum_{m=1}^\infty\frac{f(p^m)}{m}p^{-ms}\right)\quad (\Re(s)>d)\]
when the right-hand side converges.
\end{definition}

\begin{remark}\label{rem: another_def_of_AZF}
Kurokawa~\cites{kurokawa2013dualities, Deitmar-Koyama-Kurokawa_2015} introduced another definition of the absolute zeta function for a nice function $f\colon (1,\infty)\to\bbC$ by
\[\Kzeta_f(s):=\exp\left(\left.\frac{\partial}{\partial w}Z_f(w,s)\right|_{w=0}\right),\]
where
\[Z_f(w,s):=\frac{1}{\Gamma(w)}\int_1^\infty f(t)t^{-s}(\log t)^{w-1}\frac{dt}{t}.\]
It is one of the advantages of this definition that we can regard Barnes' multiple gamma function as an absolute zeta function.
Moreover, this definition is consistent with Soul\'e's since $\Kzeta_f(s)=\zeta_f(s)$ for any Puiseux polynomial $f$.
\end{remark}

Let $X=(X, \mathcal{O}_{X})$ be a monoid scheme of finite type and $X_\bbZ$ be the $\bbZ$-lift of $X$ (see \cites{Deitmar2005, Deitmar2006}). Then, Connes and Consani showed that
\[\#X_\bbZ(\bbF_q)=\sum_{x\in X}(q-1)^{r_{x}}\prod_{j=1}^{l_x}\gcd(q-1, t_{x,j})\]
for any prime power $q$ (see Propositions \ref{prop: rational-pts} and \ref{prop: rational-points_monoid-scheme}), where the non-negative integers $r_{x}$, $l_x$ and the positive integers $t_{x,j}$ are taken so that $\sts{X,x}^\times\cong\bbZ^{r_{x}}\times\prod_{j=1}^{l_x}\bbZ/t_{x,j}\bbZ$ with $t_{x,j}\mid t_{x,j+1}$ for each $x\in X$.
By using the Fourier expansion of the periodic function $\gcd(q-1, t_{x,j})$ in $q$, Deitmar, Koyama and Kurokawa \cite{Deitmar-Koyama-Kurokawa_2015}*{pp.~61--63} interpolated $\#X_\bbZ(\bbF_q)$ to a certain continuous function $N_{X_\bbZ}$ on $[1,\infty)$ and then obtained the following result.

\begin{theorem}[\cite{Deitmar-Koyama-Kurokawa_2015}*{Theorem 2.1}]\label{thm: DKK-AZF}
For the above function $N_{X_\bbZ}$, it holds that
\[\zeta_{N_{X_\bbZ}}(s)=\prod_{k=0}^{R_X}(s-k)^{\sum_{x\in X}T_x(-1)^{r_{x}-k+1}\binom{r_{x}}{k}},\]
where $T_x:=\prod_{j=1}^{l_x}t_{x,j}$ and $R_X:=\max_{x\in X}r_x$.
Moreover, if $X_\bbZ$ is a smooth projective variety of relative dimension $d$, it holds that $N_{X_\bbZ}(1)=\chi_{\mathrm{top}}(X_\bbZ(\bbC))$ and $\zeta_{N_{X_\bbZ}}(d-s)=(-1)^{\chi_{\mathrm{top}}(X_\bbZ(\bbC))}\zeta_{N_{X_\bbZ}}(s)$, where $\chi_{\mathrm{top}}(X_\bbZ(\bbC))$ is the Euler characteristic of the complex manifold $X_\bbZ(\bbC)$.
\end{theorem}

\begin{remark}
In \cite{Deitmar-Koyama-Kurokawa_2015}, Deitmar, Koyama and Kurokawa took $t_{x,j}$'s as prime powers instead of the above integers satisfying $t_{x,j}\mid t_{x,j+1}$.
\end{remark}

Despite this simple result, the proof of Theorem \ref{thm: DKK-AZF} involves relatively complicated calculations. In fact, comparing with the equation (\ref{eq: Soule-limit}) and Theorem \ref{thm: DKK-AZF}, we see that the absolute zeta function $\zeta_{N_{X_\bbZ}}(s)$ of $N_{X_\bbZ}$ coincides with the absolute zeta function $\zeta_{\CPoly_{X_\bbZ}}(s)$ of the polynomial
\[\CPoly_{X_\bbZ}(t)=\sum_{x\in X}T_x(t-1)^{r_x}.\]
This polynomial $\CPoly_{X_\bbZ}$ is characterized as the \emph{ceiling polynomial} of $X_\bbZ$ (cf. Lemma \ref{lem: existence_ceiling-floor} and Theorem \ref{thm: Noet-scheme_ceiling-floor}), which is defined as the unique polynomial in $\bbR[t]$ satisfying the following conditions:
\begin{enumerate}
    \item The inequality $\CPoly_{X_\bbZ}(q) \geq \#X_\bbZ(\bbF_{q})$ holds for every prime power $q$.
    \item There exist infinitely many prime powers $q$ such that $\CPoly_{X_\bbZ}(q) = \#X_\bbZ(\bbF_{q})$.
\end{enumerate}
Thus, we have a more simple way to obtain the above absolute zeta function $\zeta_{N_{X_\bbZ}}(s)$, not using the periodicity of $\gcd(q-1, t_{x,j})$. This simple observation is notable for us to extend Soul\'e's idea to a more general scheme of finite type over $\bbZ$ for which we do not have any formula like Connes-Consani's formula of $\#X_\bbZ(\bbF_q)$.

Similarly, by replacing $\geq$ with $\leq$ in the first condition, we can recover the polynomial
\[\FPoly_{X_\bbZ}(t)=\sum_{x\in X}(t-1)^{r_x},\]
introduced by Deitmar \cite{Deitmar2006}*{Theorem 1}. We call it the \emph{floor polynomial} of $X_\bbZ$.

The above conditions satisfied by the ceiling polynomial suggest that it is not necessary to interpolate the whole sequence $(\#X_\bbZ(\bbF_q))_q$ for the definition of an absolute zeta function of $X_\bbZ$, at least in view of the result of Deitmar-Koyama-Kurokawa \cite{Deitmar-Koyama-Kurokawa_2015}.
Therefore, it is more natural to start from a general (separated) scheme of finite type over $\bbQ$ instead of the $\bbZ$-lift of a monoid scheme of finite type.
Moreover, since the polynomial condition is too strict for most schemes of finite type over $\bbZ[S^{-1}]$, we generalize the ceiling polynomial by means of Puiseux polynomial.
For example, a desired Puiseux polynomial exists uniquely for every elliptic curve $E$ over $\bbQ$ as follows; this fact leads us to a provisional definition of the absolute zeta function of $E$.

\begin{theorem}[Corollary \ref{cor: CPoly-AZF_elliptic-curve}]\label{thm: main-2}
Let $E$ be an elliptic curve defined over $\bbQ$. Then, the Puiseux polynomial $\CPoly_E(t):=t+2t^{1/2}+1$ is characterized as the unique element in $\bbR[t^{1/\infty}]=\bigcup_{n\in\bbN}\bbR[t^{1/n}]$ satisfying the following condition: for any separated scheme $\mathcal{E}$ of finite type over $\bbZ$ satisfying that $\mathcal{E}_\bbQ\cong E$, there exists a finite set $S_{\mathcal{E}}$ of prime numbers such that for any finite set $S$ of prime numbers containing $S_{\mathcal{E}}$, the Puiseux polynomial $\CPoly_E$ satisfies the following conditions:
\begin{enumerate}
    \item The inequality $\CPoly_E(p^m) \geq \#\mathcal{E}(\bbF_{p^m})$ holds for every prime power $p^m$, where $p\not\in S$.
    \item There exist infinitely many prime powers $p^m$ such that $p\not\in S$ and the equality $\lfloor \CPoly_E(p^m)\rfloor = \#\mathcal{E}(\bbF_{p^m})$ holds.
    \item $\CPoly_E(1) \in \bbZ$.
\end{enumerate}
Moreover, the absolute zeta function of $\CPoly_E$ is
\[\zeta_{\CPoly_E}(s)=\frac{1}{s\left(s-\frac{1}{2}\right)^2(s-1)}.\]
\end{theorem}

We call $\CPoly_E$ the \emph{ceiling Puiseux polynomial} of $E$.
A drawback of $\CPoly_E$ is that the special value $\CPoly_E(1)$ does not coincide with the Euler characteristic of the complex torus $E(\bbC)$. This is not consistent with the well-known philosophy (cf.\ \cite{soule2004}*{The\'or\`eme 2}, \cite{kurokawa2005}*{Remark 2}, \cite{Deitmar2006}*{p.~141}) that the value at $1$ of the original function $f$ of the absolute zeta function $\zeta_f$ associated with a scheme coincides with its Euler characteristic.
Indeed, if $X$ is a monoid scheme of finite type such that $T_x=1$ for each $x\in X$ and $X_\bbZ$ is a smooth projective variety, then it holds that $N_{X_\bbZ}(1)=\CPoly_{X_\bbZ}(1)=\FPoly_{X_\bbZ}(1)=\chi_{\mathrm{top}}(X_\bbZ(\bbC))$, where $\chi_{\mathrm{top}}(X_\bbZ(\bbC))$ is the Euler characteristic of $X_\bbZ(\bbC)$.

On the other hand, by replacing $\geq$ in (1) (resp.\ $\lfloor \CPoly_E(p^m)\rfloor = \#\mathcal{E}(\bbF_{p^m})$ in (2)) in Theorem \ref{thm: main-2} with $\leq$ (resp.\ $\lceil \CPoly_E(p^m)\rceil = \#\mathcal{E}(\bbF_{p^m})$), we can naturally define the \emph{floor Puiseux polynomial} of $E$ and determine it as follows.

\begin{theorem}[Corollary \ref{cor: CPoly-AZF_elliptic-curve}]\label{thm: main-2_floor}
Let $E$ be an elliptic curve defined over $\bbQ$. Then, the floor Puiseux polynomial $\FPoly_{E}(t)$ of $E$ coincides with $t-2t^{1/2}+1$ and its absolute zeta function is
\[\zeta_{\FPoly_E}(s)=\frac{\left(s-\frac{1}{2}\right)^2}{s(s-1)}.\]
\end{theorem}

\begin{remark}
According to Theorems \ref{thm: main-2} and \ref{thm: main-2_floor}, it holds that
\[\zeta_{\CPoly_E}(s)=\left(\frac{1}{s\left(s-\frac{1}{2}\right)}\right)^{\otimes 2}\quad\text{and}\quad \zeta_{\FPoly_E}(s)=\left(\frac{s}{s-\frac{1}{2}}\right)^{\otimes 2},\]
where $\otimes$ denotes the tensor product that we replace $m(\rho_1,\ldots,\rho_r)$ to $-m(\rho_1,\ldots,\rho_r)$ in the definition of the Kurokawa tensor product in \cite{kurokawa1992multiple}*{p.~219}. These are compatible with the factorizations $\CPoly_E(t)=(t^{1/2}+1)^2$ and $\FPoly_{E}(t)=(t^{1/2}-1)^2$.
\end{remark}

Here, note that the special value $\FPoly_{E}(1)$ coincides with the Euler characteristic of $E(\bbC)$, which is consistent with the above philosophy.
In this view, it is fair to say that $\zeta_{\FPoly_E}$ is better than $\zeta_{\CPoly_E}$.

The organization of this article is as follows. In \S\ref{section: DKK}, we introduce ceiling polynomials and give another interpretation of \cite{Deitmar-Koyama-Kurokawa_2015}*{Theorem 2.1}. Then, we give some examples of ceiling (resp.\ floor) polynomials of specific schemes over $\bbZ[S^{-1}]$, where $S$ is a finite subset of prime numbers.
In \S\ref{section: Q-scheme}, we extend ceiling (resp.\ floor) polynomials to ceiling (resp.\ floor) Puiseux polynomials and determine the ceiling (resp.\ floor) Puiseux polynomial of an elliptic curve defined over $\bbQ$, which leads to a pair of provisional definitions of its absolute zeta function mentioned above.

\subsubsection*{Notation}
We denote the set of positive integers, non-negative integers and prime numbers by $\bbN$, $\bbN_0$ and $\primes$, respectively. We put $\PrimePow_S:=\{p^m\mid p\in\primes\setminus S,\ m\in\bbN\}$ for a subset $S$ of $\primes$ and abbreviate $\PrimePow_{\emptyset}$ to $\PrimePow$. Through this article, we write $\mathcal{P}$ as an infinite subset of $\bbN$ such as $\bbN$, $\primes$ and $\PrimePow_S$, where $S$ is finite.
For a commutative ring $R$, an $R$-algebra $A$ and a scheme $\mathcal{X}$ over $R$, the base change $\mathcal{X}\times_{\Spec R}\Spec A$ is denoted by $\mathcal{X}_A$.

%
%
%
%
%
\section{Ceiling/Floor polynomials}\label{section: DKK}
%
%
%
%
%

In this section, we review basic facts about monoid schemes and introduce the ceiling and floor polynomials of a scheme of finite type over $\bbZ[S^{-1}]$, where $S$ is a finite subset of $\primes$. After that, we characterize the polynomial $\CPoly_{X_\bbZ}$ in \S\ref{section:Intro} as the ceiling polynomial of $X_\bbZ$.

In this article, we refer to a \emph{monoid} as a commutative multiplicative monoid with the unit element $1$ and the absorbing element $0$ which maps any element to 0 by multiplication. We denote the category of monoids, abelian groups and commutative $R$-algebras with the unit element by $\Mo$, $\Abel$ and $\Alg_R$, respectively, where $R$ is a commutative ring.

\subsection{Monoid schemes}

A \emph{monoid scheme} is a topological space together with a sheaf of monoids, which is constructed by gluing spectra of monoids just like a scheme (see \cite{Deitmar2005}, where monoid schemes are called schemes over $\fun$). 

We denote the spectrum of a monoid $M$ as $\MSpec M$.\footnote{In this article, we use ``$\MSpec$'' for the spectrum of a monoid to distinguish it from ``$\Spec$'', the spectrum of a commutative ring.} Let $X$ be a monoid scheme with an affine covering $X=\bigcup_{i\in I}\MSpec M_i$. 
We say $X$ to be \emph{of finite type} if it has a covering by finitely many affine monoid schemes $\MSpec M_i$, where each $M_i$ is finitely generated \cite{Deitmar2006}.
In addition, we define $X(M):=\Hom(\MSpec M, X)$ for each $M\in\Mo$.

Let $R$ be a commutative ring. Through the base change functor $M\mapsto M\otimes_{\fun}R:=R[M]$ from $\Mo$ to $\Alg_R$, we obtain the scheme $X_R:=\bigcup_{i\in I}\Spec(M_i\otimes_{\fun}R)$ over $R$ and call $X_R$ as the \emph{$R$-lift} of $X$. Here, the isomorphism class of $X_R$ does not depend on the choices of affine coverings of $X$ \cite{Deitmar2005}.
Note that $X$ is of finite type if and only if the $\bbZ$-lift $X_\bbZ$ is of finite type over $\bbZ$ \cite{Deitmar2006}*{Lemma 2}.

Let $\fun[\cdot]\colon\Abel\to\Mo$ be the covariant functor which send a multiplicative abelian group $G$ to a monoid $G\cup\{0\}$. We put $\bbF_{1^n}:=\fun[C_n]$, where $C_n$ is a cyclic group of order $n\in\bbN$. In particular, we abbreviate $\bbF_{1^1}$ to $\fun$.

In \cite{Deitmar2006}, Deitmar mentioned the following property of monoid schemes.

\begin{proposition}[\cite{Deitmar2006}*{p.~143}]\label{prop: rational-pts}
Let $X$ be a monoid scheme of finite type. Then, it holds that \[\#X_\bbZ(\bbF_q)=\#X(\bbF_{1^{q-1}})\] for any $q\in\PrimePow$.
In particular, the underlying set of $X$ is finite, i.e., $\#X=\#X(\fun)=\#X_\bbZ(\bbF_2)<\infty$.
\end{proposition}

Connes and Consani explicitly described the right-hand side of Proposition \ref{prop: rational-pts}. Before stating their formula, we introduce some notations used hereafter.

\begin{definition}\label{def: integers}
Let $X=(X, \mathcal{O}_{X})$ be a monoid scheme of finite type. For each $x\in X$, we define $r_x, l_x\in\bbN_0$ and $t_{x,j}\in\bbN$ as the integers satisfying
\[\sts{X,x}^\times\cong\bbZ^{r_{x}}\times\prod_{j=1}^{l_x}\bbZ/t_{x,j}\bbZ\quad\text{with}\quad t_{x,j}\mid t_{x,j+1}\]
and put $T_x:=\#(\sts{X,x}^\times)_{\mathrm{tors}}$.
Here, $\sts{X,x}^\times$ denotes the group of invertible elements of the monoid $\sts{X,x}$ and $(\sts{X,x}^\times)_{\mathrm{tors}}$ denotes its torsion subgroup.
In addition, we put $R_X:=\max_{x\in X}r_x$ and $T_X:=\prod_{x\in X}T_x$.
\end{definition}

\begin{proposition}[\cite{Connes-Consani_Compositio}*{Proposition 3.22}]\label{prop: rational-points_monoid-scheme}
Let $X$ be a monoid scheme of finite type. Then, it holds that
\[\#X(\bbF_{1^{n}})=\sum_{x\in X}n^{r_{x}}\prod_{j=1}^{l_x}\gcd(n, t_{x,j})\]
for any $n\in\bbN$.
\end{proposition}

\subsection{Ceiling/Floor polynomials}

Let $X$ be a monoid scheme of finite type.
As we explained in \S\ref{section:Intro}, Deitmar, Koyama and Kurokawa \cite{Deitmar-Koyama-Kurokawa_2015} identified the absolute zeta function of $N_{X_\bbZ}$ with the absolute zeta function of the polynomial $\CPoly_{X_\bbZ}$.
In this subsection, we characterize $\CPoly_{X_\bbZ}$ as the ceiling polynomial of $X_\bbZ$.
We firstly show the uniqueness of the ceiling (resp.\ floor) polynomial of a given sequence.

\begin{lemma} \label{lem: existence_ceiling-floor}
Let $\mathcal{P}$ be an infinite subset of $\bbN$ and $\boldsymbol{A} = (A_{n})_{n \in \mathcal{P}}$ be a sequence in $\bbZ$. Then, there exists at most one polynomial $f(t) \in \bbR[t]$
satisfying the following conditions:
\begin{enumerate}
    \item The inequality $f(n) \geq A_{n}$ (resp.\ $f(n) \leq A_{n}$) holds for every $n \in \mathcal{P}$.
    \item There exist infinitely many $n \in \mathcal{P}$ such that $f(n) = A_{n}$.
\end{enumerate}
\end{lemma}

\begin{proof}
Suppose that $f, g \in \bbR[t]$ satisfy both of the conditions.
Then, since $f-g$ is a polynomial, we have the following three possibilities:
\begin{itemize}
    \item There exists $N \in \bbN$ such that $f(n)-g(n) > 0$ for every $n > N$.
    \item There exists $N \in \bbN$ such that $f(n)-g(n) < 0$ for every $n > N$.
    \item $f(n)-g(n) = 0$ for every $n \in \bbN$, i.e., $f = g$ in $\bbR[t]$.
\end{itemize}
In the first case, since $g$ (resp.\ $f$) satisfies the first condition, the inequality $f(n) > g(n) \geq A_{n}$ (resp. $A_{n}\geq f(n) > g(n)$) holds for every $n > N$, which contradicts that $f$ (resp.\ $g$) satisfies the second condition.
By changing the roles of $f$ and $g$, we see that the second case is also impossible.
Thus, we obtain the conclusion.
\end{proof}

\begin{definition}\label{def: ceiling-floor_sequence}
When the polynomial $f$ in Lemma \ref{lem: existence_ceiling-floor} exists, we call the unique polynomial $f$ the \emph{ceiling} (resp.\ \emph{floor}) \emph{polynomial} of $\boldsymbol{A}$.
\end{definition}

\begin{definition}
Let $S$ be a proper subset of $\primes$ and $\mathcal{X}$ be a scheme of finite type over $\bbZ[S^{-1}]$. We call the ceiling (resp.\ floor) polynomial of $\left(\#\mathcal{X}(\bbF_{q})\right)_{q\in\PrimePow_S}$ the \emph{ceiling} (resp.\ \emph{floor}) \emph{polynomial} of $\mathcal{X}$ and denote it by $\CPoly_{\mathcal{X}}$ (resp.\ $\FPoly_{\mathcal{X}}$).
\end{definition}

According to Propositions \ref{prop: rational-pts} and \ref{prop: rational-points_monoid-scheme}, we obtain the ceiling (resp.\ floor) polynomial of the $\bbZ[S^{-1}]$-lift of a monoid scheme of finite type.

\begin{theorem} \label{thm: Noet-scheme_ceiling-floor}
Let $X$ be a monoid scheme of finite type and $S$ be a finite subset of $\primes$. Set $\mathcal{X}:=X_{\bbZ[S^{-1}]}$,
\[e_{x,j,S}:=\begin{cases}1 &\text{if $2\mid t_{x,j}$ and $2\in S$,}\\0 &\text{otherwise,}\end{cases}\quad\text{and}\quad T_{x,S}:=\prod_{j=1}^{l_x}2^{e_{x,j,S}}.\]
Then, it holds that
\[\CPoly_{\mathcal{X}}(t)=\sum_{x\in X} T_x(t-1)^{r_{x}}\in \bbZ[t]\quad\text{and}\quad \FPoly_{\mathcal{X}}(t)=\sum_{x\in X} T_{x,S}(t-1)^{r_{x}}\in \bbZ[t].\]
In particular, $\CPoly_{\mathcal{X}}$ is independent of $S$. Moreover, it holds that
\[\zeta_{\CPoly_{\mathcal{X}}}(s)=\prod_{k=0}^{R_X}(s-k)^{\sum_{x\in X}T_x(-1)^{r_{x}-k+1}\binom{r_{x}}{k}}\]
and
\[\zeta_{\FPoly_{\mathcal{X}}}(s)=\prod_{k=0}^{R_X}(s-k)^{\sum_{x\in X}T_{x,S}(-1)^{r_{x}-k+1}\binom{r_{x}}{k}}.\]
\end{theorem}

\begin{proof}
At first, we consider the polynomial $\CPoly_{\mathcal{X}}$. The first condition in Lemma \ref{lem: existence_ceiling-floor} follows from the inequality $\gcd(n-1, t_{x, j}) \leq t_{x, j}$ for any $n\in\bbN$. We can check the second condition by applying Dirichlet's theorem on arithmetic progressions to the prime numbers $p$ such that $p\equiv 1\pmod{T_X}$. Thus, the polynomial $\sum_{x\in X} T_x(t-1)^{r_{x}}$ coincides with $\CPoly_{\mathcal{X}}$.

Next, we consider the polynomial $\FPoly_{\mathcal{X}}$. Let $T_X'$ be the odd integer satisfying $T_X=2^{e}T_X'$ for some $e\in\bbN_0$. The first condition follows from the inequality $\gcd(q-1, t_{x,j}) \geq 2^{e_{x,j,S}}$ for any $x\in X$, $j\in\{1,\ldots,l_x\}$ and $q\in\PrimePow_S$. The second condition in the case where $2\not\in S$ follows from the fact that $2^{\varphi(T_X') k+1}-1\equiv 1\pmod{T_X'}$ for any $k\in\bbN$, where $\varphi$ is Euler's totient function. In the case where $2\in S$, we see that there are infinitely many $p\in\primes\setminus S$ such that $p\equiv 2 \pmod{T_X'}$ and $p\equiv 3 \pmod{4}$ by combining Dirichlet's theorem on arithmetic progression and the Chinese remainder theorem. We denote the set of such $p$'s by $P$. For $p\in P$, it holds that $\gcd(p-1, T_X) = 2$ (resp.\ $1$) when $T_X$ is even (resp.\ odd), and hence $\gcd(p-1, t_{x,j}) = 2^{e_{x,j,S}}$ for any $x\in X$ and $j\in\{1,\ldots,l_x\}$.
Thus, the second condition follows.

The equation on the absolute zeta function follows from the equation (\ref{eq: Soule-limit}) and the calculation of $\CPoly_{\mathcal{X}}$ and $\FPoly_{\mathcal{X}}$.
\end{proof}

\begin{remark}
Let $X = (X, \mathcal{O}_{X})$ be a monoid scheme of finite type. Then,
\[\sum_{x\in X} T_x(t-1)^{r_{x}} \in \bbZ[t]\quad \left(\text{resp.}\ \sum_{x\in X} (t-1)^{r_{x}}\in \bbZ[t]\right)\]
is the ceiling (resp.\ floor) polynomial of $\left(\#X(\bbF_{1^{n-1}})\right)_{n\in\bbN\cap[2,\infty)}$ by Proposition \ref{prop: rational-points_monoid-scheme} and a similar argument of the proof of Theorem \ref{thm: Noet-scheme_ceiling-floor}.
In fact, the floor polynomial of $\left(\#X(\bbF_{1^{n-1}})\right)_{n\in\bbN\cap[2,\infty)}$ coincides with the polynomial $N(x)$ introduced by Deitmar in \cite{Deitmar2006}*{Theorem 1} since it satisfies the condition therein and such a polynomial is unique.
\end{remark}

Theorem \ref{thm: Noet-scheme_ceiling-floor} shows that $\zeta_{\CPoly_{\mathcal{X}}}(s)$ coincides with $\zeta_{N_{X_\bbZ}}(s)$ in Theorem \ref{thm: DKK-AZF}, which Deitmar, Koyama and Kurokawa obtained in \cite{Deitmar-Koyama-Kurokawa_2015} using the Fourier expansion. Thus, $\zeta_{N_{X_\bbZ}}(s)$ is an invariant of $X_{\bbZ[S^{-1}]}$ independent of $S$, and hence it is an invariant of its generic fiber $X_\bbQ$ (cf.\ Example \ref{eg: monoid_scheme}).

\subsection{Other examples of ceiling/floor polynomials}\label{subsection: examples_ceiling-floor}

We give some examples of the ceiling (resp.\ floor) polynomials of other specific schemes over $\bbZ[S^{-1}]$, especially those of relative dimension $1$.

\begin{theorem}\label{thm: proj-line_1}
Let $n\in\bbN$, $\mathcal{A}_n:=\bbA^1_\bbZ\setminus\{0,1,\ldots,n-1\}$ and $S$ be a finite subset of $\primes$.
Set $n_1:=\min_{p\in\primes\setminus S}\{p,n\}$. Then, it holds that
\[\CPoly_{\mathcal{A}_{n,\bbZ[S^{-1}]}}(t)=t-n_1\quad\text{and}\quad\FPoly_{\mathcal{A}_{n,\bbZ[S^{-1}]}}(t)=t-n.\]
\end{theorem}

\begin{proof}
This follows from the fact that
\[\#\mathcal{A}_{n,\bbZ[S^{-1}]}(\bbF_q)=q-\#(\bbF_p\cap\{0,1,\ldots,n-1\})=q-\min\{p,n\}\]
for each $q=p^m\in\PrimePow_S$.
\end{proof}

Let $n\geq 2$. Replacing $\{0,1,\ldots,n-1\}$ with $\{0\}\cup\mu_{n-1}$, where $\mu_{n-1}$ is the set of the $(n-1)$-th roots of unity, we can obtain the following result.

\begin{theorem}\label{thm: proj-line_2}
Let $n\in\bbN\cap[2,\infty)$, $\mathcal{G}_n:=\bbA^1_\bbZ\setminus\left(\{0\}\cup\mu_{n-1}\right)=\mathbb{G}_{m,\bbZ}\setminus\mu_{n-1}$ and $S$ be a finite subset of $\primes$.
Set
\[n_2:=\begin{cases}3 &\text{if $2\nmid n$ and $2\in S$,}\\2 &\text{otherwise.}\end{cases}\]
Then, it holds that \[\CPoly_{\mathcal{G}_{n,\bbZ[S^{-1}]}}(t)=t-n_2\quad\text{and}\quad\FPoly_{\mathcal{G}_{n,\bbZ[S^{-1}]}}(t)=t-n.\]
\end{theorem}

\begin{proof}
This follows from Theorem \ref{thm: Noet-scheme_ceiling-floor} and the fact that $\mu_{n-1}$ is the $\bbZ$-lift of $\MSpec\bbF_{1^{n-1}}$.
\end{proof}

We give another example of ceiling (resp.\ floor) polynomials. Let $\mathcal{C}^{\Delta}$ be the Pell conic of discriminant $\Delta\ne 0$, defined as an affine curve over $\bbZ$ defined by
\begin{align*}  
    \begin{cases}x^2-\frac{\Delta}{4}y^2=1 & \text{if $\Delta\equiv 0\bmod{4}$,}\\x^2+xy+\frac{1-\Delta}{4}y^2=1 & \text{if $\Delta\equiv 1\bmod{4}$.}\end{cases}
\end{align*}
Then, the number of the $\bbF_q$-rational points of $\mathcal{C}^{\Delta}$ is given as follows.

\begin{theorem}\label{thm: Pell-curve}
Let $q=p^m\in\PrimePow$. Then,
\[\#\mathcal{C}^{\Delta}(\bbF_q)=\begin{cases}q-\left(\frac{\Delta}{p}\right)^m & \text{if $p\ne 2$, $p\nmid \Delta$,} \\ 2q & \text{if $p\ne 2$, $p\mid \Delta$,} \\ q-(-1)^{\frac{\Delta^2-1}{8}m} & \text{if $p=2$, $2\nmid \Delta$,} \\ q & \text{if $p=2$, $2\mid \Delta$,}\end{cases}\]
where $\left(\frac{\Delta}{p}\right)$ is the Legendre symbol.
Moreover, let $S_\Delta$ be the set of prime numbers dividing $\Delta$. For any finite subset $S$ of $\primes$, it holds that
\[\CPoly_{\mathcal{C}^{\Delta}_{\bbZ[S^{-1}]}}(t)=\begin{cases}2t & \text{if $S_\Delta\setminus\{2\}\not\subset S$,}\\t+1 & \text{if $\Delta$ is not a square and $S_\Delta\setminus\{2\}\subset S$,}\\t-1 & \text{if $\Delta$ is a square and $S_\Delta\subset S$,}\\t & \text{if $\Delta$ is an even square, $S_\Delta\setminus\{2\}\subset S$ and $2\not\in S$,}\end{cases}\]
and \[\FPoly_{\mathcal{C}^{\Delta}_{\bbZ[S^{-1}]}}(t)=t-1.\]
\end{theorem}

\begin{proof}
Assume that $p\ne 2$ and $p\nmid \Delta$.
If $\Delta\bmod{p}\in\bbF_q^{\times 2}$, then we have $\#\mathcal{C}^{\Delta}(\bbF_q)=q-1$ since it holds that
\[\mathcal{C}^{\Delta}(\bbF_q)\cong\bbF_q^\times; (x,y)\mapsto x+\frac{\sqrt{\Delta}}{2}y.\]
If $\Delta\bmod{p}\in\bbF_q^\times\setminus\bbF_q^{\times 2}$, then it holds that
\[\mathcal{C}^{\Delta}(\bbF_{q})\cong\mathop{\mathrm{Ker}}\left(N_{\bbF_{q^2}/\bbF_q}\colon\bbF_{q^2}^\times\to\bbF_q^\times\right); (x,y)\mapsto x+\frac{\sqrt{\Delta}}{2}y\]
and the norm map $N_{\bbF_{q^2}/\bbF_q}$ is surjective. Therefore, we have $\#\mathcal{C}^{\Delta}(\bbF_q)=\#\bbF_{q^2}^\times/\#\bbF_{q}^\times=q+1$. Thus, it holds that $\#\mathcal{C}^{\Delta}(\bbF_q)=q-\left(\frac{\Delta}{p}\right)^m$ if $p\ne 2$ and $p\nmid \Delta$.

Assume that $p\ne 2$ and $p\mid \Delta$, then \[\#\mathcal{C}^{\Delta}(\bbF_{q})=\#\{(x,y)\in\bbF_q\times\bbF_q\mid x^2\equiv 4\pmod{p}\}=2q.\]

Assume that $p=2$ and $p\nmid \Delta$.
If $\Delta\equiv 1\pmod{8}$, then
we have $\#\mathcal{C}^{\Delta}(\bbF_q)=q-1$ since
\[\mathcal{C}^{\Delta}(\bbF_q)\cong\bbF_q^\times; (x,y)\mapsto x.\]
If $\Delta\equiv 5\pmod{8}$ and $m$ is even, then we have $\#\mathcal{C}^{\Delta}(\bbF_q)=q-1$ since
\[\mathcal{C}^{\Delta}(\bbF_q)\cong\bbF_q^\times; (x,y)\mapsto x+\zeta_3 y,\]
where $\zeta_3\in\bbF_q$ denotes a primitive third root of unity.
If $\Delta\equiv 5\pmod{8}$ and $m$ is odd, then we have $\#\mathcal{C}^{\Delta}(\bbF_q)=q+1$ since
\[\mathcal{C}^{\Delta}(\bbF_{q})\cong\mathop{\mathrm{Ker}}N_{\bbF_{q^2}/\bbF_q}; (x,y)\mapsto x+\zeta_3 y,\]
where $\zeta_3\in\bbF_{q^2}$ denotes a primitive third root of unity.

Assume $p=2$ and $p\mid \Delta$, then \[\#\mathcal{C}^{\Delta}(\bbF_q)=\#\{(x,y)\in\bbF_q\times\bbF_q\mid x^2\equiv 1\pmod{2}\}=q.\]

The statements on the ceiling and floor polynomials of $\mathcal{C}^{\Delta}_{\bbZ[S^{-1}]}$ follow from the above calculation of $\#\mathcal{C}^{\Delta}(\bbF_q)$.
\end{proof}

Next, it is natural to study the ceiling (resp.\ floor) polynomial of a curve $\mathcal{C}$ of positive genus defined over $\bbZ[S^{-1}]$. According to Theorem \ref{thm: Pell-curve}, one can expect that the ceiling polynomial crucially depends on the bad reductions of $\mathcal{C}$ and becomes more simple if $\mathcal{C}$ is smooth over $\bbZ[S^{-1}]$.
However, the following result shows that the ceiling polynomial does not exist even for an elliptic curve defined over $\bbZ[S^{-1}]$ whenever $S$ is finite.

\begin{proposition} \label{prop: ceiling_non-existence}
Let $S$ be a finite subset of $\primes$ and $\mathcal{E}$ be an elliptic curve defined over $\bbZ[S^{-1}]$.
Then, there exists no ceiling or floor polynomial of $\mathcal{E}$.
\end{proposition}

\begin{proof}
By Hasse's theorem, it holds that
\[
	\#\mathcal{E}(\bbF_{p}) < p+1+2\sqrt{p}
\]
for every prime number $p\in\primes\setminus S$. On the other hand, the Sato-Tate conjecture \cites{BLGHT,Clozel-Harris-Taylor} implies that for every $\epsilon > 0$, there exist prime numbers $p\in\primes\setminus S$ such that
\[
	\#\mathcal{E}(\bbF_{p}) > p+1+2\sqrt{p}(1-\epsilon).
\]
These facts imply that there exists no ceiling polynomial $\CPoly_{\mathcal{E}}$ of $\mathcal{E}$.
Indeed, if such a polynomial $\CPoly_{\mathcal{E}}$ exists, then the Sato-Tate conjecture and the first condition in Lemma \ref{lem: existence_ceiling-floor} imply that
\[
	\forall \alpha > 0, \ \exists N_{0} \in \bbN \ \text{s.t.} \ \forall p\in\primes\setminus S,\ ( p > N_{0} \Rightarrow \CPoly_{\mathcal{E}}(p) > p+\alpha ).
\]
However, since $\CPoly_{\mathcal{E}}$ is a polynomial, the above estimate is equivalent to the following:
\[
	\exists \delta > 0 \ \text{s.t.} \ \exists N_{1} \in \bbN \ \text{s.t.} \ \forall p\in\primes\setminus S,\ ( p > N_{1} \Rightarrow \CPoly_{\mathcal{E}}(p) > (1+\delta)p ).
\]
Since the inequality $(1+\delta)p > p+1+2\sqrt{p}$ holds for every $p \gg 1$, Hasse's theorem implies that
\[
	\exists N_{2} \in \bbN \ \text{s.t.} \ \forall p\in\primes\setminus S,\ ( p > N_{2} \Rightarrow \CPoly_{\mathcal{E}}(p) > \#\mathcal{E}(\bbF_{p}) ),
\]
which contradicts the second condition in Lemma \ref{lem: existence_ceiling-floor}.

The non-existence of the floor polynomial $\FPoly_{\mathcal{E}}$ of $\mathcal{E}$ follows from a similar argument.
\end{proof}

%
%
%
%
%
\section{Ceiling/Floor Puiseux polynomials}\label{section: Q-scheme}
%
%
%
%
%

In this section, we introduce ceiling (resp.\ floor) Puiseux polynomials by replacing the polynomial condition in Lemma \ref{lem: existence_ceiling-floor} by means of Puiseux polynomials. Then, after introducing the ceiling (resp.\ floor) Puiseux polynomial of a separated scheme of finite type over $\bbQ$, we identify the ceiling (resp.\ floor) Puiseux polynomial of an elliptic curve over $\bbQ$ as the Puiseux polynomial $t+2t^{1/2}+1$ (resp.\ $t-2t^{1/2}+1$).

\subsection{Ceiling/Floor Puiseux polynomials}

We begin with the definition of the ceiling (resp.\ floor) Puiseux polynomial of a general integer sequence.

\begin{definition}
Let $R$ be a commutative ring.
We define $R[t^{1/\infty}]$ as the residue ring of the polynomial ring $R\left[t_{n}\mid n \in \bbN\right]$ in countably many indeterminates $t_{n}$ modulo the ideal $I$ generated by $t_{mn}^{m}-t_{n}$ for all $m, n \in \bbN$, and set $t^{1/n} := t_{n} \bmod{I}$.
We call an element of $R[t^{1/\infty}]$ a Puiseux polynomial with coefficients in $R$.
\end{definition}

Suppose that $R = \bbR$ (or its subring). Then, by fixing a branch $1^{1/n} = 1$ for each $n \in \bbN$, each Puiseux polynomial in $\PuiPoly$ defines a continuous function on $\bbR_{\geq 0}$ to $\bbR$.
In what follows, we identify each Puiseux polynomial with this function.
Similarly to Lemma \ref{lem: existence_ceiling-floor}, we show the uniqueness of a certain Puiseux polynomial.

\begin{lemma} \label{lem: existence_ceiling-floor-Puiseux}
Let $\mathcal{P}$ be an infinite subset of $\bbN$ and $\boldsymbol{A} = (A_{n})_{n \in \mathcal{P}}$ be a sequence in $\bbZ$.
Then, there exists at most one Puiseux polynomial $f(t) \in \PuiPoly$
satisfying the following conditions:
\begin{enumerate}
    \item The inequality $f(n) \geq A_{n}$ (resp.\ $f(n) \leq A_{n}$) holds for every $n \in \mathcal{P}$.
    \item There exist infinitely many $n \in \mathcal{P}$ such that the equality $\lfloor f(n)\rfloor = A_{n}$ (resp.\ $\lceil f(n)\rceil = A_{n}$) holds.
    \item $f(1) \in \bbZ$.
\end{enumerate}
\end{lemma}

\begin{proof}
Suppose that $f, g \in \PuiPoly$ satisfy both of the conditions. Then, since $f-g$ is a Puiseux polynomial, it is a polynomial of $t^{1/n}$ for some $n \in \bbN$. Hence, we have the following three possibilities:
\begin{itemize}
    \item There exists some $N \in \bbN$ such that $f(n)-g(n) \geq 1$ for every $n > N$.
    \item There exists some $N \in \bbN$ such that $f(n)-g(n) \leq -1$ for every $n > N$.
    \item $f-g$ is a constant in the open interval $(-1, 1)$.
\end{itemize}
In the first case, since $g$ (resp.\ $f$) satisfies the first condition, the inequality $f(n) \geq g(n)+1 \geq A_{n}+1$ (resp.\ $g(n)\leq f(n)-1\leq A_{n}-1$) holds for every $n > N$, which contradicts that $f$ (resp.\ $g$) satisfies the second condition.
By changing the roles of $f$ and $g$, we see that the second case is also impossible. In the third case, it holds that $f=g$ since $f(1)-g(1)=0$ by the third condition.
Thus, we obtain the conclusion.
\end{proof}

\begin{definition}\label{def: ceiling-Puiseux_sequence}
When the Puiseux polynomial $f$ in Lemma \ref{lem: existence_ceiling-floor-Puiseux} exists, we call the unique Puiseux polynomial $f$ the \emph{ceiling} (resp.\ \emph{floor}) \emph{Puiseux polynomial} of $\boldsymbol{A}$.
\end{definition}

If there exists a polynomial with integral coefficients satisfying the conditions in Lemma \ref{lem: existence_ceiling-floor}, then it satisfies the conditions in Lemma \ref{lem: existence_ceiling-floor-Puiseux}. In this sense, the Puiseux polynomial in Lemma \ref{lem: existence_ceiling-floor-Puiseux} is a generalization of the polynomials with integral coefficients in Lemma \ref{lem: existence_ceiling-floor}, which contain polynomials which have been studied in the context of absolute zeta functions (e.g., \cites{soule2004, Deitmar2006, Deitmar-Koyama-Kurokawa_2015}).

As we mentioned after Theorem \ref{thm: Pell-curve}, we can expect a more simple ceiling Puiseux polynomial if the information on pathological prime numbers is excluded. Hence, it is fair to define a ceiling (resp.\ floor) Puiseux polynomial of an algebraic variety over $\bbQ$ (and more generally a separated scheme of finite type over $\bbQ$) as follows.

\begin{definition}\label{def: ceiling-Puiseux_Q-sch}
Let $X$ be a separated scheme of finite type over $\bbQ$. Assume that there exists a Puiseux polynomial $f$ satisfying the following condition: for any separated scheme $\mathcal{X}$ of finite type over $\bbZ$ satisfying that $\mathcal{X}_\bbQ\cong X$, there exists a finite subset $S_{\mathcal{X}}$ of $\primes$ such that for any finite subset $S$ of $\primes$ containing $S_{\mathcal{X}}$, the Puiseux polynomial $f$ is the ceiling (resp.\ floor) Puiseux polynomial of $\left(\#\mathcal{X}(\bbF_q)\right)_{q\in\PrimePow_{S}}$. Then, we call $f$ the \emph{ceiling} (resp.\ \emph{floor}) \emph{Puiseux polynomial} of $X$ and denote it by $\CPoly_X$ (resp.\ $\FPoly_X$).
\end{definition}

The following facts are useful for verification of the uniqueness of the ceiling and floor Puiseux polynomials of $X$ and their practical calculation.

\begin{theorem}[\cite{Serre2012_counting-fct}*{Theorems 4.12 and 4.13}]
Let $\mathcal{X}$ be a separated scheme of finite type over $\bbZ$ and $l$ be a prime number.
Then, there exists a finite subset $\Sigma$ of $\bbP$ (independent of $l$) such that for every $p \in \primes \setminus (\Sigma \cup \{ l \})$ and every $m\in\bbN$, the following equality holds:
\[\#\mathcal{X}(\bbF_{p^m})=\sum_{i=0}^{2\dim \mathcal{X}_{\bbQ}}(-1)^i\mathop{\mathrm{Tr}}(\sigma_p^{-m}\mid H_c^i(\mathcal{X}_{\overline{\bbQ}},\bbQ_l)),\]
where $\sigma_p$ is the $p$-th power Frobenius automorphism in $\mathop{\mathrm{Gal}}(\overline{\bbF_p}/\bbF_p)$, which acts on $H_c^i(\mathcal{X}_{\overline{\bbQ}},\bbQ_l)$ via the specialization map $H_c^i(\mathcal{X}_{\overline{\bbF_p}},\bbQ_l)\overset{\sim}{\to}H_c^i(\mathcal{X}_{\overline{\bbQ}},\bbQ_l)$.
\end{theorem}

\begin{corollary}\label{cor: single_model_is_enough}
Let $\mathcal{X}, \mathcal{Y}$ be separated schemes of finite type over $\bbZ$ such that $\mathcal{X}_{\bbQ}\cong\mathcal{Y}_{\bbQ}$. Then, there exists a finite subset $\Sigma'$ of $\primes$ such that the following equality holds for every $q\in\PrimePow_{\Sigma'}$:
\[\#\mathcal{X}(\bbF_{q}) = \#\mathcal{Y}(\bbF_{q}). \]
In particular, in the setting of Definition \ref{def: ceiling-Puiseux_Q-sch}, if $f$ is the ceiling (resp. floor) Puiseux polynomial of $(\#\mathcal{X}(\bbF_{q}))_{q \in \bbP_{S}^{\bbN}}$ for some $\mathcal{X}$ and for some $S_{\mathcal{X}}$ with an arbitrary $S\supset S_{\mathcal{X}}$, then it coincides with the ceiling (resp. floor) Puiseux polynomial of $X$.
\end{corollary}

According to this corollary, it is sufficient to verify the condition in Definition \ref{def: ceiling-Puiseux_Q-sch} not for all $\mathcal{X}$ but for a single $\mathcal{X}$. Moreover, the ceiling and floor Puiseux polynomials for such $\mathcal{X}$ are unique respectively if they exist. Using this fact, we obtain the ceiling and floor Puiseux polynomials for the generic fibers of specific schemes which appeared in \S\ref{section: DKK} as follows.

\begin{eg}[cf.~\cite{Deitmar2006}*{Proposition 4.3}]\label{eg: monoid_scheme}
Let $X$ be a monoid scheme of finite type such that $X_\bbZ$ is separated. Thus, it holds that
\[\CPoly_{X_\bbQ}(t)=\sum_{x\in X} T_x(t-1)^{r_{x}}\quad\text{and}\quad\FPoly_{X_\bbQ}(t)=\sum_{x\in X} T_{x,\{2\}}(t-1)^{r_{x}}\]
by Theorem \ref{thm: Noet-scheme_ceiling-floor} and Corollary \ref{cor: single_model_is_enough}. Indeed, it is sufficient to take $\mathcal{X}=X_\bbZ$ and $S_{\mathcal{X}}=\{2\}$. In particular, it holds that $\CPoly_{X_\bbQ}=\FPoly_{X_\bbQ}$ if and only if the torsion subgroup of $\sts{X, x}^{\times}$ is 2-torsion for all $x\in X$.
\end{eg}

\begin{eg}
Put $X=\mathcal{A}_{n,\bbQ}$. By Theorem \ref{thm: proj-line_1} and Corollary \ref{cor: single_model_is_enough}, it holds that \[\CPoly_{\mathcal{A}_{n,\bbQ}}(t)=\FPoly_{\mathcal{A}_{n,\bbQ}}(t)=t-n.\]
Indeed, it is sufficient to take $\mathcal{X}=\mathcal{A}_{n}$ and $S_{\mathcal{X}}$ as the set of prime numbers less than $n$.
\end{eg}

\begin{eg}
Put $X=\mathcal{G}_{n,\bbQ}$. By Theorem \ref{thm: proj-line_2} and Corollary \ref{cor: single_model_is_enough}, it holds that \[\CPoly_{\mathcal{G}_{n,\bbQ}}(t)=t-2\quad\text{and}\quad\FPoly_{\mathcal{G}_{n,\bbQ}}(t)=t-n.\]
Indeed, it is sufficient to take $\mathcal{X}=\mathcal{G}_{n}$ and $S_{\mathcal{X}}=\{2\}$. In particular, it holds that $\CPoly_{\mathcal{G}_{n,\bbQ}}=\FPoly_{\mathcal{G}_{n,\bbQ}}$ if and only if $n=2$.
\end{eg}

\begin{eg}
Put $X=\mathcal{C}^{\Delta}_\bbQ$. By Theorem \ref{thm: Pell-curve} and Corollary \ref{cor: single_model_is_enough}, it holds that
\[\CPoly_{\mathcal{C}^{\Delta}_\bbQ}(t)=\begin{cases}t-1 & \text{if $\Delta$ is a square,}\\t+1 & \text{if $\Delta$ is not a square,}\end{cases}\quad\text{and}\quad \FPoly_{\mathcal{C}^{\Delta}_\bbQ}(t)=t-1.\]
Indeed, it is sufficient to take $\mathcal{X}=\mathcal{C}^{\Delta}$ and $S_{\mathcal{X}}$ as the set of prime numbers dividing $2\Delta$.
In particular, it holds that $\CPoly_{\mathcal{C}^{\Delta}_\bbQ}=\FPoly_{\mathcal{C}^{\Delta}_\bbQ}$ if and only if $\Delta$ is square, which is equivalent to $\mathcal{C}^{\Delta}_\bbQ\cong \mathbb{G}_{m,\bbQ}$.
Note that even if $\Delta$ is not a square, the scalar extension (base change) $\mathcal{C}^{\Delta}_\bbQ\otimes\bbQ(\sqrt{\Delta})$ can be identified with the $\bbQ(\sqrt{\Delta})$-lift of the monoid scheme $\mathbb{G}_{m, \fun}$.
\end{eg}

\subsection{Ceiling/Floor Puiseux polynomial of a projective curve and its maximal/minimal reduction}\label{subsection: curve}

Let $C$ be a smooth proper curve over $\bbQ$ which is geometrically irreducible of genus $g>0$. Then, by the spreading out principle (see \cite{poonen2017rational}*{Theorem 3.2.1}), there exist a finite subset $S_{C}$ of $\primes$ and a smooth proper scheme $\mathcal{C}$ of finite type over $\bbZ[S_{C}^{-1}]$ such that $\mathcal{C}_{\bbQ} \cong C$.

For $q=p^m\in\PrimePow_{S_{C}}$, the Hasse-Weil bound (see \cite{Serre2012_counting-fct}*{\S4.7.2.2}) implies that
\[q-2g\sqrt{q}+1 \leq \#\mathcal{C}(\bbF_q) \leq q+2g\sqrt{q}+1.\]
The closed fiber $\mathcal{C}_{\bbF_p}$ of $\mathcal{C}$ is called \emph{$\bbF_q$-maximal} (resp.\ \emph{$\bbF_q$-minimal}) if $\#\mathcal{C}(\bbF_q)$ attains the Hasse-Weil upper (resp.\ lower) bound, i.e.,
\[\#\mathcal{C}(\bbF_q) = q+2g\sqrt{q}+1\quad (\mathrm{resp.}\ \#\mathcal{C}(\bbF_q) = q-2g\sqrt{q}+1).\]
In view of the ceiling (resp. floor) Puiseux polynomial, we are interested in the distribution of the prime powers $q$ for which $\mathcal{C}_{\bbF_p}$ is $\bbF_q$-maximal (resp. $\bbF_q$-minimal). By the definition of the ceiling (resp.\ floor) Puiseux polynomial of $C$, we obtain the following proposition.

\begin{proposition}\label{prop: ceiling-floor_curve}
Assume that there exist infinitely many prime numbers $p\in\primes\setminus S_{C}$ for which $\mathcal{C}_{\bbF_p}$ is $\bbF_{p^m}$-maximal (resp.\ $\bbF_{p^m}$-minimal) for some $m\in\bbN$. Then, it holds that
\[\CPoly_C(t)=t+2gt^{1/2}+1\quad \left(\text{resp. }\FPoly_C(t)=t-2gt^{1/2}+1\right).\]
\end{proposition}

\subsection{Ceiling/Floor Puiseux polynomial of an elliptic curve}\label{subsection: elliptic_curve}

Let $E$ be an elliptic curve defined over $\bbQ$. Like \S\ref{subsection: curve}, there exist a finite subset $S_{E}$ of $\primes$ and an elliptic curve $\mathcal{E}$ over $\bbZ[S_{E}^{-1}]$ such that $\mathcal{E}_{\bbQ} \cong E$.
Then, the following fact is known concerning a supersingular elliptic curve.

\begin{lemma}[\cite{arithmetic_elliptic-curve}*{p.~155}]\label{lem: supersingular_equiv}
Suppose that $p\in\primes\setminus (S_{E}\cup\{2,3\})$. Then, the following conditions are equivalent:
\begin{enumerate}
    \item $\mathcal{E}_{\bbF_p}$ is supersingular, i.e., $\#\mathcal{E}(\bbF_p) = p + 1$.
    \item $\mathcal{E}_{\bbF_p}$ is $\bbF_{p^{4k-2}}$-maximal and $\bbF_{p^{4k}}$-minimal for any $k\in\bbN$, i.e., $\#\mathcal{E}(\bbF_{p^{4k-2}})=p^{4k-2}+2p^{2k-1}+1$ and $\#\mathcal{E}(\bbF_{p^{4k}})=p^{4k}-2p^{2k}+1$.
    \item $\mathcal{E}_{\bbF_p}$ is $\bbF_{p^2}$-maximal.
    \item It holds that \[Z(\mathcal{E}_{\bbF_p}, T):=\exp\left(\sum_{m=1}^\infty\frac{\#\mathcal{E}(\bbF_{p^m})}{m}T^m\right)=\frac{1+pT^2}{(1-T)(1-pT)}.\]
\end{enumerate}
\end{lemma}

\begin{proof}
Let $\alpha$ be an eigenvalue of the $p$-th power Frobenius endomorphism on the Tate module of $E$. Then, it holds that
\begin{equation}\label{eq: trace-formula}
    \#\mathcal{E}(\bbF_{p^m}) = 1 - \left(\alpha^m + \frac{p^m}{\alpha^m}\right) + p^m \tag{$\star$}
\end{equation}
for any $m\in\bbN$ (see e.g.\ \cite{arithmetic_elliptic-curve}*{Theorem 2.3.1}). In particular, by specializing it to $m=1$, the equivalence $(1)\Leftrightarrow\alpha^2=-p$ follows. The equation $\alpha^2=-p$ is equivalent to (2) and (3), respectively.
Moreover, the equivalence $(1)\Leftrightarrow (4)$ follows since
\begin{align*}
    Z(\mathcal{E}_{\bbF_p}, T) &= \exp\left(\sum_{m=1}^\infty\left(1-\left(\alpha^m+\frac{p^m}{\alpha^m}\right)+p^m\right)\frac{T^m}{m}\right)\\
    &=\frac{(1-\alpha T)(1-\frac{p}{\alpha}T)}{(1-T)(1-pT)}=\frac{1+\left(\#\mathcal{E}(\bbF_p)-p-1\right)T+pT^2}{(1-T)(1-pT)}.
\end{align*}
\end{proof}

Proposition \ref{prop: ceiling-floor_curve} and Lemma \ref{lem: supersingular_equiv} $(1)\Leftrightarrow (2)$ lead us to the natural question whether there exist infinitely many prime numbers $p$ such that $\mathcal{E}_{\bbF_p}$ is supersingular. The answer is known to be affirmative due to Elkies as follows.

\begin{theorem}[\cite{Elkies1987}]\label{thm: Elkies}
Let $E$ be an elliptic curve over $\bbQ$. Then, there exist infinitely many prime numbers at which $E$ has good supersingular reduction.
\end{theorem}

\begin{remark}
In fact, Elkies~\cite{Elkies_number-field} obtained a similar result for every elliptic curve over an arbitrary number field $F$ (of finite degree) which has at least one field embedding $F\subset \bbR$.
\end{remark}

As the consequence of Theorem \ref{thm: Elkies} and Lemma \ref{lem: supersingular_equiv} $(1)\Leftrightarrow (2)$, we see that there exist infinitely many prime numbers $p\in\primes\setminus S_E$ for each of which $\mathcal{E}_{\bbF_p}$ is $\bbF_{p^m}$-maximal for some $m\in\bbN$. Therefore, we can determine the ceiling (resp.\ floor) Puiseux polynomial of an elliptic curve defined over $\bbQ$ as follows.

\begin{corollary}\label{cor: CPoly-AZF_elliptic-curve}
Let $E$ be any elliptic curve over $\bbQ$. Then, it holds that
\[\CPoly_E(t)=t+2t^{1/2}+1\quad\text{and}\quad \FPoly_{E}(t)= t-2t^{1/2}+1.\]
Moreover, the absolute zeta functions of $\CPoly_E$ and $\FPoly_E$ are
\[\zeta_{\CPoly_E}(s)=\frac{1}{s\left(s-\frac{1}{2}\right)^2(s-1)}\quad\text{and}\quad \zeta_{\FPoly_E}(s)=\frac{\left(s-\frac{1}{2}\right)^2}{s(s-1)},\]
respectively.
\end{corollary}

\begin{remark}
If $X$ is a monoid scheme of finite type whose $\bbZ$-lift is a smooth projective variety, then Deitmar, Koyama and Kurokawa deduced the equality
\[\#X(\fun)=N_{X_\bbZ}(1)=\chi_{\mathrm{top}}(X_\bbZ(\bbC))\]
from the Weil conjecture for $X_{\bbF_p}$ (cf.\ the proof of \cite{Deitmar-Koyama-Kurokawa_2015}*{Theorem 2.1}). In fact, we could formally obtain the similar equation \[``\#\mathcal{E}(\fun)"=0=\chi_{\mathrm{top}}(\mathcal{E}(\bbC))\]
if we substituted $m=0$ in the equation (\ref{eq: trace-formula}) in the proof of Lemma \ref{lem: supersingular_equiv}, which is the consequence of the Weil conjecture for $\mathcal{E}_{\bbF_p}$.
Moreover, the Puiseux polynomial $\FPoly_E$ satisfies that \[\FPoly_E(1)=\chi_{\mathrm{top}}(E(\bbC))=\chi_{\mathrm{top}}(S^{1} \times S^{1}).\]
These observations are all consistent with the philosophy that the number of ``the $\fun$-rational points'' of a scheme and the value at $1$ of the original function $f$ of the absolute zeta function $\zeta_f$ associated with it coincide with its Euler characteristic (cf.\ \cite{soule2004}*{Th\'eor\`eme 2}, \cite{kurokawa2005}*{Remark 2}, \cite{Deitmar2006}*{p.~141}).
On the other hand, the Puiseux polynomial $\CPoly_E$ is not consistent with the above philosophy.
In this view, it is fair to say that $\zeta_{\FPoly_E}$ is better than $\zeta_{\CPoly_E}$.
\end{remark}

\begin{remark}
According to \cite{Charles_pair_of_elliptic-curve}, for any pair of elliptic curves $E_1$, $E_2$ over a number filed $K$, there are infinitely many prime ideals of $K$ at which the reductions of $E_1$ and $E_2$ are geometrically isogenous. This might suggest that all elliptic curves over $K$ are ``geometrically isogenous over $\fun$" in some sense. On the other hand, if $K=\bbQ$, then Corollary \ref{cor: CPoly-AZF_elliptic-curve} shows that both $\CPoly_E$ and $\FPoly_E$ are determined purely in terms of the Betti numbers of the topological 2-dimensional torus $S^{1} \times S^{1}$. In particular, they are independent of the isogeny class of $E$. This might even suggest that all elliptic curves over $\bbQ$ are ``isogenous over $\fun$" at least in view of Tate's isogeny theorem over $\bbF_p$ (see e.g.\ \cite{arithmetic_elliptic-curve}*{III.7.7}).
\end{remark}

\appendix

%
%
%
%
%
\section{Ceiling/Floor Puiseux polynomial of an elliptic curve in the case of $\mathcal{P}=\primes\setminus S$}
%
%
%
%
%

Let $S$ be a finite subset of $\primes$. In this appendix, we discuss the ceiling and floor Puiseux polynomials of the sequence $\left(\#\mathcal{E}(\bbF_p)\right)_{p\in\primes\setminus S}$ instead of the sequence $\left(\#\mathcal{E}(\bbF_q)\right)_{q\in\PrimePow_S}$ in \S\ref{section: Q-scheme}. As a result, in the case of elliptic curves defined over $\bbQ$ with complex multiplication, we obtain the same Puiseux polynomial as its ceiling and floor Puiseux polynomial.

\begin{definition}
Let $X$ be a separated scheme of finite type over $\bbQ$. Assume that there exists a Puiseux polynomial $f$ satisfying the following condition: for any separated scheme $\mathcal{X}$ of finite type over $\bbZ$ satisfying that $\mathcal{X}_\bbQ\cong X$, there exists a finite subset $S_{\mathcal{X}}$ of $\primes$ such that for any finite subset $S$ of $\primes$ containing $S_{\mathcal{X}}$, the Puiseux polynomial $f$ is the ceiling (resp.\ floor) Puiseux polynomial of $\left(\#\mathcal{X}(\bbF_p)\right)_{p\in\primes\setminus S}$. Then, we call $f$ the \emph{prime ceiling} (resp.\ \emph{floor}) \emph{Puiseux polynomial} of $X$ and denote it by $\CPoly_X'$ (resp.\ $\FPoly_X'$).
\end{definition}

\begin{remark}
Comparing it with Definition \ref{def: ceiling-Puiseux_Q-sch}, the first condition in Lemma \ref{lem: existence_ceiling-floor-Puiseux} gets weaker and the second one gets stronger for $\boldsymbol{A}=\left(\#\mathcal{X}(\bbF_p)\right)_{p\in\primes\setminus S}$ than for $\boldsymbol{A}=\left(\#\mathcal{X}(\bbF_q)\right)_{q\in\PrimePow_S}$.
\end{remark}

Let $E$ be an elliptic curve defined over $\bbQ$. As mentioned in \S\ref{subsection: elliptic_curve}, there exist a finite subset $S_{E}$ of $\primes$ and an elliptic curve $\mathcal{E}$ over $\bbZ[S_{E}^{-1}]$ such that $\mathcal{E}_{\bbQ} \cong E$. Then, for $p\in\primes\setminus S_{E}$, the Hasse bound implies that
\[
	p+1-2\sqrt{p} < \#\mathcal{E}(\bbF_{p}) < p+1+2\sqrt{p}.
\]
Then, $p$ is called a \emph{champion} (resp.\ \emph{trailing}) prime if the equality \[\#\mathcal{E}(\bbF_{p}) = p+1+\lfloor 2\sqrt{p}\rfloor \quad\left(\text{resp.}\ \#\mathcal{E}(\bbF_{p}) = p+1- \lceil 2\sqrt{p} \rceil\right)\] holds~\cite{James-Pollack}.
Let $\pi_{E}^{+}$ (resp.\ $\pi_{E}^{-}$) be the set of champion (resp.\ trailing) prime numbers for $E$ and $\pi_{E}^{\pm}(x) := \pi_{E}^{\pm} \cap (0,x]$ for every $x \in (0,\infty)$. Then, the following is obvious:

\begin{proposition}[cf.~Proposition \ref{prop: ceiling-floor_curve}]\label{prop: elliptic-curve_prime}
Assume that $\#\pi_{E}^{\pm}=\infty$, then it holds that $\CPoly'_{E}=\CPoly_{E}$ and $\FPoly'_{E}=\FPoly_{E}$.
\end{proposition}

For a CM elliptic curve over $\bbQ$, the following fact on $\pi_{E}^{\pm}(x)$ is known.

\begin{theorem} [{\cite{James-Pollack}*{Theorem 1}}] \label{thm: James-Pollack}
Suppose that $E$ has complex multiplication over $\overline{\bbQ}$. Then, the following asymptotic relation holds:
\[
	\pi_{E}^{\pm}(x) \sim \frac{2}{3\pi} \cdot \frac{x^{3/4}}{\log x}
	\quad (x \to \infty).
\]
In particular, it holds that $\#\pi_{E}^{\pm}=\infty$.
\end{theorem}

According to Theorem \ref{thm: James-Pollack}, the prime ceiling (resp.\ floor) Puiseux polynomial of a CM elliptic curve coincides with the Puiseux polynomial in Proposition \ref{prop: elliptic-curve_prime}. On the other hand, for an elliptic curve defined over $\bbQ$ without complex multiplication, it is conjectured in \cite{James-Tran-Trinh-Wertheimer-Zantout}*{Conjecture 2.3} that
\[
	\pi_{E}^{\pm}(x) \sim c_{E} \cdot \frac{x^{1/4}}{\log x}
	\quad (x \to \infty),
\]
where $c_{E} \in (0,\infty)$ is a constant.
Currently, the above estimate of $\pi_{E}^{\pm}(x)$ in the case where $E$ is a non-CM elliptic curve is verified only under some assumptions such as the Generalized Riemann Hypothesis (cf.\ \cite{extremal-prime_non-CM}).

\subsection*{Acknowledgement}

The authors would like to thank Professors Kenichi Bannai, Takeshi Katsura, Masato Kurihara and Taka-aki Tanaka for carefully checking the manuscript of this article and giving many helpful comments. The authors also thank the referee for valuable comments concerning the manuscript.

%
%
%
%
%
%
%
%
%
%

\bibliography{ceiling_floor}

\end{document}